\documentclass[12pt,reqno]{amsart}

\usepackage{latexsym}
\usepackage{amssymb}

\renewcommand{\Re}{{\operatorname{Re}\,}}

\newcommand{\beq}{\begin{equation}}
\newcommand{\eeq}{\end{equation}}

\renewcommand{\epsilon}{\varepsilon}
\newcommand{\dist}{{\operatorname{dist}}}
\newcommand{\var}{{\bf Var}}

\newcommand{\szego}{Szeg\H{o} }

\newcommand{\inv}{^{-1}}
\newcommand{\kahler}{K\"ahler }
\newcommand{\sqrtk}{\sqrt{k}}
\newcommand{\wt}{\widetilde}

\newcommand{\PP}{{\mathbb P}}
\newcommand{\N}{{\mathbb N}}
\newcommand{\R}{{\mathbb R}}
\newcommand{\C}{{\mathbb C}}
\newcommand{\Q}{{\mathbb Q}}
\newcommand{\Z}{{\mathbb Z}}

\newcommand{\CP}{\C\PP}
\renewcommand{\d}{\partial}
\newcommand{\dbar}{\bar\partial}
\newcommand{\ddbar}{\partial\dbar}

\newcommand{\E}{{\mathbf E}}

\newcommand{\half}{{\textstyle \frac 12}}

\renewcommand{\phi}{\varphi}

\newcommand{\acal}{\mathcal{A}}

\newcommand{\ccal}{\mathcal{C}}
\newcommand{\dcal}{\mathcal{D}}

\newcommand{\hcal}{\mathcal{H}}
\newcommand{\ical}{\mathcal{I}}

\newcommand{\lcal}{\mathcal{L}}

\newcommand{\ncal}{\mathcal{N}}

\newcommand{\pcal}{\mathcal{P}}

\newcommand{\al}{\alpha}
\newcommand{\be}{\beta}
\newcommand{\ga}{\gamma}

\newcommand{\la}{\lambda}
\newcommand{\ep}{\varepsilon}

\newcommand{\om}{\omega}
\newcommand{\Om}{\Omega}

\newtheorem{theo}{{\sc Theorem}}[section]

\newtheorem{cor}[theo]{{\sc Corollary}}

\newtheorem{lem}[theo]{{\sc Lemma}}
\newtheorem{prop}[theo]{{\sc Proposition}}

\newenvironment{rem}{\medskip\noindent{\it Remark:\/} }{\medskip}
\newtheorem{defin}[theo]{{\sc Definition}}

\title[Variance of random zeros on complex manifolds]
{Asymptotic expansion of the variance of random zeros on complex manifolds}

\author{Bernard Shiffman}
\address{Department of Mathematics, Johns Hopkins University, Baltimore, MD
21218, USA} \email{shiffman@math.jhu.edu}

\thanks{Research  partially supported by NSF grant CCF-1640970.}

\begin{document}

\begin{abstract} 
Linear statistics of  random zero sets are  integrals of  smooth differential forms over the zero set and as such are smooth analogues of the  volume  of the random zero set inside a fixed domain.
We derive an asymptotic expansion for the variance of  linear statistics of the zero divisors of random holomorphic sections of powers of a positive line bundle on a compact \kahler manifold. This expansion sharpens the leading-order asymptotics (in the codimension one case) given by Shiffman--Zelditch in 2010.  \end{abstract}

\maketitle

\section{Introduction}

The distribution  of zeros of random polynomials has a long history going back to Bloch--Polya \cite{BP}, Kac \cite{Kac}, Littlewoood--Offord \cite{LO}, Maslova \cite{Ma} and others, who studied the zeros of random real polynomials with i.i.d.\  coefficients.  Zeros of random polynomials of one complex variable were studied by Hammersley \cite{Ham} in 1956 and in the 1990s by  Hannay \cite{Han} and others \cite{BBL, FH, NV} and in the 2000s in \cite{Bl,ST,SZeq}. In higher dimensions, zeros of $m$  simultaneous random polynomials in $\C^m$ were investigated by Edelman--Kostlan \cite{EK} in 1995 and more recently by \cite{Ba,BS, DMM, Zhu} and others.

Polynomials on $\C^m$ of degree $k$ can be identified with holomorphic sections of the $k$-th power of the hyperplane section bundle on $\CP^m$, and thus a natural generalization is to holomorphic sections of powers of a positive holomorphic line bundle on a compact \kahler manifold. The zeros of random sections of such line bundles form the subject of a series of  papers \cite{SZ,BSZ1,BSZ2,SZa,SZ5} based on the asymptotic expansion of the Bergman kernel for powers of a line bundle introduced by Catlin \cite{Cat} and  Zelditch \cite{Z} and further developed in \cite{BSZ2,LS} (see also \cite{Be,MM,Ti,Xuhao,Ze2}). Additional references for zeros of random polynomials and random holomorphic sections can be found in \cite{BCHM}. 

This article concerns the  fluctuations of the zero sets of random holomorphic sections. We study  the ``linear statistics" of random zero sets; these are the integrals of a smooth test form over the zero sets of random holomorphic sections. Linear statistics can be regarded as smooth approximations of ``discontinuous statistics" such as the volume (or cardinality in the point case) of a random zero set inside a fixed domain, as studied in \cite{SZa}. Linear statistics were studied by Forrester--Honner \cite{FH}, Sodin--Tsirelson \cite{ST} and Nazarov--Sodin \cite{NS} for functions  of one complex variable, by \cite{SZ5}  for holomorphic sections on compact \kahler manifolds, and by \cite{Ba} for polynomials in several complex variables.  We shall apply the methods of \cite{SZ5} to give a sharp asymptotic expansion (Theorem~\ref{sharp}) for the case of divisors (codimension 1), and we also apply the results of \cite{LS} to compute the sub-leading term of this expansion.

To describe our framework, we begin with  a positively curved
Hermitian holomorphic line bundle  $(L,h)\to M$ over a compact complex manifold
$M$ of dimension $m$. Recall that the curvature form $\Theta_h\in\dcal^{1,1}(M)$ of $L$ is given locally by 
$\Theta_h= -\ddbar\log h(e_L,\overline{e_L}) $, where $e_L$ is a nonvanishing local holomorphic section of $L$. We then give
$M$ the \kahler form $\om= \frac i2\Theta_h$, and we give the spaces $H^0(M,L^k)$ of holomorphic sections the Hermitian inner products induced by the \kahler form
$\om$ and the Hermitian metrics $h^k$:
\begin{equation}\label{inner}\langle s_1, \bar s_2 \rangle = \int_M h^k(s_1, \bar 
s_2)\,\frac 1{m!}\om^m\;,\qquad s_1, s_2 \in
H^0(M,L^k)\,.\end{equation}  

For a section $s^k\in H^0(M,L^k)$ (not identically zero), we let $Z_{s^k}\in\dcal'^{1,1}(M)$ denote the current of integration along the zero set (divisor) of $s^k$:
$$(Z_{s^k},\psi)= \int_{\{s^k=0\}}\psi,\qquad \psi\in\dcal^{m-1,m-1}(M)\,,$$ taking into account multiplicities. Here, $\dcal^{p,q}(M)$ denotes the space of smooth $(p,q)$-forms on the (compact) manifold $M$. In particular, if $\dim M=1$, then $Z_{s^k}= \sum_j n_j \delta_{a_j}$, where  $s^k$ has zeros of multiplicity $n_j$ at the point $a_j$. (In our setting, all multiplicities equal 1, almost surely.)

The following asymptotic expansion for the variance of linear statistics of random zeros is the main result of this article:
 
\begin{theo}\label{sharp} Let $(L,h)$ be a
positive Hermitian holomorphic line bundle over a compact \kahler
manifold $(M,\om)$ of dimension $m$, where $\om= \frac i2 \Theta_h$. We give the spaces $H^0(M,L^k)$ the standard Gaussian probability measures (see \eqref{standard}) corresponding to the inner products \eqref{inner}.

Let  $\psi\in \dcal^{m-1,m-1}_\R(M)$ be a real $\ccal^\infty$ $(m-1,m-1)$-form on $M$. Then for random sections $s^k\in H^0(M,L^k)$, the variances $\var\big(Z_{s^k},\psi\big)$ have an asymptotic expansion
\beq\var\big(Z_{s^k},\psi\big) \sim A_0k^{-m}+A_1 k^{-m-1} + \cdots +A_jk^{-m-j}+ \cdots.\label{expansion}\eeq The leading and sub-leading coefficients are given by \begin{eqnarray}\label{A0}A_0&=&\frac{\pi^{m-2}\zeta(m+2)}4\,\|\ddbar\psi\|_{2}^2\,,\\A_1&=& 
 -\pi^{m-2}\zeta(m+3)\left\{\frac 18\int_M\rho|\ddbar\psi|^2\Omega_M+\frac14\|\d^*\ddbar\psi\|_{2} ^2\right\},\label{A1}\end{eqnarray} where $\rho$ is the scalar curvature and $\Om_M=\frac 1{m!}\om^m$ is the volume form on $M$.
\end{theo}

In fact, all the coefficients $A_j$ can be given in terms of $\ddbar\psi$,  curvature invariants of $M$,  and their derivatives; see Theorem~\ref{general}.

Theorem \ref{sharp} sharpens the result in \cite{SZ5}, where it was shown that
$$\var\big(Z_{s^k},\psi\big) = A_0 k^{-m}+
 O(k^{-m-\frac 12 +\ep})\;.$$ Similar results for the variance in higher codimension were  also given in \cite{SZ5}.

In particular, for $\dim M=1$ we have:
\begin{cor} Let $M$ be a compact Riemann surface and let $(L,h)\to (M,\om)$ be as in Theorem \ref{sharp}. Then for  $\psi\in\ccal^\infty_\R(M)$,  
\beq\var\left(Z_{s^k},\psi\right)=\frac {\zeta(3)}{16\pi}\|\Delta\psi\|^2k\inv -\frac{\pi^3}{2880}\left\{\int_M\rho|\Delta\psi|^2\,\omega+{ \frac 12}\|d\Delta\psi\|_{2} ^2\right\}k^{-2}+O(k^{-3}).\label{sharpRiemann}\eeq\end{cor}
The leading term of \eqref{sharpRiemann} was obtained by Sodin and Tsirelson \cite{ST} for the cases where $M$ is $\CP^1$, the hyperbolic disk, and $\C^1$, where it was shown  that
$\var\left(Z_{s^k},\psi\right)=\frac {\zeta(3)}{16\pi}\|\Delta\psi\|^2k\inv+o(k\inv)$.

\section{Background}

We summarize here the notation and results that are used in this paper. 
We recall that the standard Gaussian measure $\ga_V$ on an $n$-dimensional complex vector space $V$  with a Hermitian inner product  is given by \beq\label{standard}d\ga_V(v)=\prod_{j=1}^n\left(\frac i{2\pi}e^{-|c_j|^2}  dc_j\wedge d\bar c_j\right)=\frac 1{\pi^n}e^{-\sum_{j=1}^n|c_j|^2} \prod_{j=1}^n \frac i2dc_j\wedge d\bar c_j\,,\quad v=\sum_{j=1}^n c_jv_j\in V,\eeq where $\{v_1,\dots, v_n\}$ is an orthonormal basis of $V$.  We let  $\E$  denote the  expected value of a  (real or vector-space valued) random variable, and we let $\var$ denote the variance of a (real) random variable.

Throughout this paper, $(L,h)\to M$ denotes a positive Hermitian holomorphic line bundle  over an $m$-dimensional compact \kahler manifold $M$ with \kahler form $\om= \frac i2\Theta_h$. We give the spaces $H^0(M,L^k)$ of holomorphic sections of $L^k=L^{\otimes k}$ the inner products given by \eqref{inner} together with  their corresponding standard Gaussian probability  measures.

Let $\{S_1^k,S_2^k,\dots S_{n_k}^k\}$ be a basis for $H^0(M,L^k)$,  $k\ge 1$. The reproducing kernel $K_k(z,w)$ of the orthogonal projection $K_k:\lcal^2(M,L^k)\to H^0(M,L^k)$ is given by the Bergman kernel
\beq\label{Bergman} K_k(z,w)= \sum_{j=1}^{n_k}S_j^k(z)\otimes \overline{S_j^k(w)}=\E\left(s^k(z)\otimes \overline{s^k(w)}\right)\in L^k_z\otimes \overline{L^k_w}\,.\eeq
The Bergman projection $K_k$ lifts to an orthogonal projection $\Pi_k:\lcal^2(X)\to \hcal^2_k(X)$, where   $X=\{\la\in L^*:\|\la\|=1\}$ is the circle bundle  of the dual line bundle $L^*$, and  $\hcal^2_k(X)= \{f\in \hcal^2(X): f(e^{i\theta}x)= e^{ik\theta}f(x)\}$, where $\hcal^2(X)$ is the Hardy space of CR-holomorphic $\lcal^2$ functions on $X$ (see \cite{Z,SZ2,SZa}). The reproducing kernel for $\Pi_k$ is the \szego kernel $\Pi_k\left((z,\theta_1),(w,\theta_2)\right) =e^{i(\theta_1-\theta_2)}\Pi_k(z,w)$ with norm given by
$$|\Pi_k(z,w)|= \|K_k(z,w)\|_{h^k(z)\otimes h^k(w)}\,.$$

It was shown in \cite{SZ} using the Poincar\'e--Lelong formula, \beq\label{PL} Z_{s^k}= \frac i{2\pi} \ddbar \log\|s^k\|^2_{h^k} +\frac k\pi\om\,,\eeq  that the expected value of the zero current of a Gaussian random section $s^k\in H^0(M,L^k)$ is given by
\beq\label{probPL} \E Z_{s^k}=\frac i{2\pi} \ddbar \log \Pi_k(z,z) +\frac k\pi \om\,.\eeq

On the other hand, the variance of $Z_{s^k}$ depends on the   {\it normalized \szego kernel\/}  introduced in \cite{SZa}:
\begin{equation}\label{Pk} P_k(z,w)=
\frac{|\Pi_k(z,w)|}{\Pi_k(z,z)^\frac 12 \Pi_k(w,w)^\frac
12}\;.\end{equation} We note that $P_k(z,z)=1$ and  $0\le P_k(z,w)\le 1$ by Cauchy-Schwarz.
The following variance formula from \cite{SZa} involves the normalized kernel $P_k(z,w)$ together with the function 
\begin{equation}\label{G}  G(t):= -\frac 1{4\pi^2}
\int_0^{t} \frac{\log(1-s)}{s}\,ds\ =\ \frac 1{4\pi^2}
\sum_{n=1}^\infty\frac{t^{n}}{n^2}\;,
\qquad -1\le t\le 1.\end{equation}

\begin{theo} \label{bipot} {\rm \cite[Theorem 3.1]{SZa}} Let $(L,h)\to(M,\om)$ be as in
Theorem
\ref{sharp}. Let $Q_k:M\times M\to [0,+\infty)$ be
the function given  by
\begin{equation}
\label{Qk} Q_k(z,w)=  G\left(P_k(z,w)^2\right) \,. \end{equation}  
Then the variance of the linear statistics $(Z_{s_k},\psi)$ of a random section $s_k\in H^0(M,L^k)$ is given by \begin{equation} \label{varint1} \var( Z_{s^k},\psi)= \int_{M\times
M} Q_k(z,w)\,i\ddbar \psi(z)\wedge i\ddbar \psi(w)
\;,\end{equation} for all real $(m-1,m-1)$-forms $\psi$ on $M$
with $\ccal^2$ coefficients.\end{theo}

\begin{rem} The function $\wt G(t)$ in \cite{SZa} equals $G(t^2)$ here. Although Theorem~\ref{bipot} involves $G(t)$ only  for $0\le t\le 1$, we allow $t$ to take negative values in the proof of Theorem~\ref{sharp}.\end{rem}

To obtain the asymptotic formula of Theorem~\ref{sharp}, we shall apply the asymptotics of the Bergman--\szego kernel to formula \eqref{varint1}.  To do this, we must overcome the fact that $G(t)$ is not differentiable at $t=P_k(z,z)=1$.

As in \cite{SZa,SZ5}, we shall apply the off-diagonal
asymptotics  of $\Pi_k(z,w)$ and $P_k(z,w)$ from \cite{SZ2,SZa}.   We recall that the \szego kernel $\Pi_k$ decays rapidly away from the diagonal as $k$ increases:

\begin{prop}\label{DPdecay} {\rm \cite[Prop.~2.6]{SZa}} Let $(L,h)\to
(M,\om)$ be as in Theorem
\ref{sharp}, and let $p>0$. 
we have
$$ P_k(z,w)=O(k^{-p})\qquad \mbox{uniformly for }\ \dist(z,w)\ge
\left[(2p+1)\frac {\log k}{k}\right]^{1/2} \;.$$
\end{prop}

\begin{cor}\label{Qdecay} {\rm \cite[Lemma~3.4]{SZa}} Under the hypotheses of Proposition \ref{DPdecay}, we have $$ Q_k(z,w)=O(k^{-p})\qquad \mbox{uniformly for }\ \dist(z,w)\ge
\left[(p+1)\frac {\log k}{k}\right]^{1/2} \;.$$
\end{cor}

\begin{proof} The bound follows immediately from \eqref{Qk}, Proposition \ref{DPdecay}, and the fact that $G(t)=O(t)$ for $0\le t\le 1$. \end{proof}

\begin{rem} The normalized \szego kernel satisfies the sharper decay rate away from the diagonal: $$P_k(z,w)=O(e^{-c\sqrtk\, \dist(z,w)}).$$
(See, for example, \cite[Th.~2.5]{B} and \cite{Li}.)  However, Proposition \ref{DPdecay} suffices for our purposes.\end{rem}

Corollary \ref{Qdecay} allows us to replace the integral in \eqref{varint1} with an integral over a small shrinking neighborhood of the diagonal, as we shall demonstrate below.  We then shall apply the ``near-diagonal" asymptotic expansion of the \szego kernel from \cite{SZ2,LS}. To give the asymptotic expansion in a neighborhood $U$ of a point $z_0\in M$, we let  $\Phi_{z_0}:U\buildrel{\approx}\over\to U'\subset\C^m$ be a local coordinate chart with $\Phi_{z_0}(z_0)=0$ and we write, by abuse of notation, $$z_0+u \equiv\Phi_{z_0}\inv(u)\in U\,,\quad\mbox{for }\  u\in U'\,.$$

\begin{theo} \label{near}{\rm \cite[Theorem~3.1]{SZ2}} Let $(L,h)\to (M,\om)$ be as in Theorem~\ref{sharp}, and let $z_0\in M$ and  $\ep,b\in\R^+$. Then using normal coordinates at $z_0$,
$$\Pi_k\left(z_0+\frac{u}{\sqrtk},z_0+\frac{v}{\sqrtk}\right)= \frac{k^m}{\pi^m}e^{u\cdot \bar v-\frac12(|u|^2+|v|^2)}\left[1+ \sum_{r = 1}^{n} k^{-\frac r2}
p_{r}(u,v) + k^{-\frac{(n +1)}2} E_{kn}(u,v)\right]\;,$$
where the $p_r$ are polynomials in $(u,\bar u, v,\bar v)$ (depending on $z_0$ and the choice of coordinates) of degree $\le 5r$ and of the same parity as $r$, and
\beq\label{est96}\|D^jE_{kn}(u,v)\|\le C_{jn\ep b}k^{\ep}\quad \mbox{for }\
|u|+|v|<b\sqrt{\log k}\,,\eeq for $j,n\ge 0$. 
Furthermore, if one chooses smoothly varying  coordinate charts $\Phi_z$ for $z$ varying in a closed domain $\overline V$, then
the constants  $C_{jn\ep b}$ can be chosen independently of $z_0\in\overline V$. 
\end{theo}

Here $\|D^jF(u,v)\|$ denotes the sum of the norms of  the derivatives of $F$ of order $j$ with respect to  $u,\bar u,v,\bar v$.  The estimate \eqref{est96} is equivalent to equation
(96) in \cite{SZ2}, where the result was shown to hold for
almost-complex symplectic manifolds. (A short proof of Theorem~\ref{near} is given in \cite[\S 5]{SZa};  see \cite{BBS, HKSX} for alternative approaches. For real-analytic metrics, the expansion holds for $|u|+|v| < \ep k^{1/4}$; see  \cite{HLX}.)

We assume henceforth that $M$ is covered by a finite collection of open sets $V$ such that there are smoothly varying coordinate charts $\Phi_z$ for $z\in \overline V$. This will guarantee uniformity of remainder estimates for $z_0\in M$.

Although the terms of the expansion in Theorem~\ref{near} depend on the choice of coordinates, if one uses specific coordinates introduced by Bochner \cite{Bo}, then the terms of the expansion are well described in terms of curvature invariants, as was given in \cite{LS}. In particular, choosing a  nonvanishing holomorphic section $e_L$ of $L$ over a neighborhood of $z_0$, we let $\phi=-\log h(e_L,\overline{e_L})$ so that $\phi$ is the \kahler potential:
\beq\label{phi} \om = \frac i2\ddbar\phi = \frac i2 \sum_{p,q=1}^m g_{p\bar q}dz^p\wedge d \bar z^{\bar q}\,, \qquad g_{p\bar q}= \frac{\d^2\phi}{\d z^p\d \bar z^{\bar q}}\,.\eeq
We recall that the components of the curvature tensor
are given by 
\begin{equation}\label{Rijkl}
R_{p\bar qj\bar l}=-\frac{\d ^2 g_{p\bar q}}{\d  z^j\d  \bar z^{l}}+\sum_{r,s=1}^mg^{r\bar s}\,\frac{\d  g_{p\bar s}}{\d  z^j}\,\frac{\d  g_{r\bar q}}{\d \bar  z^{l}}\,.
\end{equation}  The scalar curvature $\rho$ is given by   \beq\rho=\sum_{p,q,j,l}g^{p\bar q}g^{j\bar l}R_{p\bar qj\bar l}\,. \eeq
We write $$R_{z_0}(u,\bar v,u,\bar v)=\sum_{p,q,j,l}R_{p\bar qj\bar l}(z_0)\,u^p\bar v^q u^j\bar v^l\,,$$ for  $u=(u^1,\dots, u^m)$, $v=(v^1,\dots, v^m)$.

We shall use the following result from \cite{LS} to obtain the expressions for $A_0$ and $A_1$ in \eqref{A0}--\eqref{A1}. 

\begin{theo} \label{p1p2}{\rm \cite[Theorem~2.3]{LS}} Assuming the hypotheses of Theorem~\ref{near}, suppose that the   local holomorphic section $e_L$ and local coordinates at $z_0$ are chosen   such that the \kahler potential $\phi=-\log h(e_L,\overline{e_L})$ is of the form
\begin{equation} \label{Korder4}\phi(z_0+z) = |z|^2 + \frac 14 \sum_{p,q,j,l}\frac{\d^4\phi}{\d z^p\d z^j\d\bar z^q\d\bar z^l}(z^0) \,z^pz^j\bar z^q \bar z^l + O(|z|^5)\,.\end{equation} 
Then \begin{eqnarray}\label{p1} p_1(u,v) &=&0\,, \\ \label{p2} p_2(u,v)&=&\frac12\rho(z_0)+\frac 18 R_{z_0}(u,\bar u,u,\bar u) + \frac 18 R_{z_0}(v,\bar v,v,\bar v) -\frac 14 R_{z_0}(u,\bar v,u,\bar v)\,.\end{eqnarray}\end{theo}
Note that \eqref{Korder4} implies that the coordinates are normal coordinates at $z_0$.

\section{The leading and sub-leading terms}\label{begin}

We shall prove in this section  the following weaker form of Theorem \ref{sharp}:

\begin{prop} \label{sharp2} Assuming the hypotheses of Theorem \ref{sharp}, we have for all $\ep>0$,
\begin{multline*}\var\big(Z_{s^k},\psi\big)=\frac{\pi^{m-2}}{k^m}\left[\frac{\zeta(m+2)}4\,\|\ddbar\psi\|_{2}^2 \right.\\ \left.- \zeta(m+3)\left\{\frac 18\int_M\rho|\ddbar\psi|^2\Omega_M+\frac14\|\d^*\ddbar\psi\|_{2} ^2\right\}k\inv +O(k^{-3/2+\ep})\right].\end{multline*}\end{prop}

To prove Theorem~\ref{sharp2}, we shall derive off-diagonal asymptotics of $Q_k(z,w)$ (Lemma~\ref{Qapprox}). To obtain these asymptotics, we first apply Theorem \ref{near}:
\begin{lem}\label{near2} Under the hypotheses of Theorem \ref{near}, there are constants $c_{n\ep b}$ independent of $z_0$ such that
$$P_k\big(z_0+\frac u{\sqrtk},z_0\big)^2 = e^{-|u|^2}\left(1+\sum_{r = 1}^{n} k^{-r/2}
a_{r}(u) + k^{-(n +1)/2}\wt E_{kn}(u)\right),$$
where the $a_r(u)$ are polynomials in $(u,\bar u)$ of degree $\le 5r$, of the same parity as
$r$, and with $a_r(0)=0$, $da_r(0)=0$, and  $$|\wt E_{kn}(u)|\le |u|^2 c_{n\ep b}k^{\ep}\quad \mbox{for }\
|u|<b\sqrt{\log k}\,,$$ for
 $n\ge 1$.
\end{lem}

\begin{proof} The result with the weaker remainder estimate \beq\label{weaker}\textstyle P_k\big(z_0+\frac u{\sqrtk},z_0\big)^2 = e^{-|u|^2}\left(1+\sum_{r = 1}^{n} k^{-r/2}
a_{r}(u) + k^{-(n +1)/2}\wt E_{kn}(u)\right),\quad|\wt E_{kn}(u)|\le  C_{n\ep b}k^{\ep}\eeq is an immediate consequence of Theorem~\ref{near} with $j=0$. To obtain the sharper estimate,  let $$\la_k(u):=  P_k\big(z_0+\frac u{\sqrtk},z_0\big)^2\,.$$ We first show that $a_n(0)=\wt E_{kn}(0)=0$. Let $n\ge 1$ and suppose by induction that $a_j(0)=0$ for $j\le n-1$.  We then have by the weaker estimate \eqref{weaker}, $$1=\la_k(0)=1+k^{-n/2}a_n(0)+k^{-(n+1)/2}\wt E_{kn}(0),$$ and hence  $\wt E_{kn}(0)=-a_n(0)k^{1/2}$. But $\wt E_{kn}(0) = O(k^\ep)$; therefore $a_n(0)=0$ and hence $\wt E_{kn}(0)=0$.

Since $\la_k(u)\le \la_k(0)=1$ by Cauchy-Schwarz, we have $d\la_k(0)=0$. To show that $da_n(0)=0$, we suppose  by induction that $da_j(0)=0$ for $j\le n-1$. Then  
$$0=d\la_k(0)=k^{-n/2}da_n(0)+k^{-(n+1)/2}d\wt E_{kn}(0),$$ and 
 it follows by the same argument as above that  $da_n(0)=d\wt E_{kn}(0)=0$. (Thus the polynomial $a_n$ has no constant or linear terms.)

Since $d\wt E_{kn}(0)=0$, we have by Theorem \ref{near} $$\left|\frac{\d }{\d u^ j}\wt E_{kn}(u)\right| \le
\int_0^1\left| \sum_{l=1}^m\left[u^ l\frac{\d^2 \wt E_{kn}}{\d u^ l\d u^ j}(tu)+\bar u^ l\frac{\d^2 \wt E_{kn}}{\d\bar u^ l\d u^ j}(tu)\right]\right|dt\le  \,C_{2n\ep b}\,|u|\,k^\ep\,,$$ for  $|u|<b\sqrt{\log k}$. Thus
$$|\wt E_{kn}(u)| \le \int_0^12\left| \sum_{j=1}^mu^ j\frac{\d \wt E_{kn}}{\d u^ j}(tu)\right|dt\le 2\sqrt  m \,C_{2n\ep b}\,|u|^2\,k^\ep\,,$$ for  $|u|<b\sqrt{\log k}$.

\end{proof}
 
\begin{theo} \label{withK} Under the hypotheses of Theorem \ref{p1p2}, there are constants $c_{n\ep b}$ independent of $z_0$ such that
$$P_k\big(z_0+\frac u{\sqrtk},z_0\big)^2 = e^{-|u|^2}\left[1+\frac14R_{z_0}(u,\bar u,u,\bar u)k^{-1}+\sum_{r = 3}^{n} k^{-r/2}
a_{r}(u) + k^{-(n +1)/2}\wt E_{kn}(u)\right],$$
where the $a_r$ are polynomials in $(u,\bar u)$ of degree
$\le 5r$ of the same parity as
$r$ and with $a_r(0)=0$, $da_r(0)=0$, and $$|\wt E_{kn}(u)|\le |u|^2 c_{n\ep b}k^{\ep}\quad \mbox{for }\
|u|<b\sqrt{\log k}\,,$$ for
 $n\ge 3$.
\end{theo}
 
\begin{proof} 
By Lemma \ref{near2} and \eqref{p1}, we have \begin{eqnarray*} P_k\big(z_0+\frac u{\sqrtk},z_0\big)^2 &=&\frac{\Pi_k\big(z_0+\frac u{\sqrtk},z_0\big)\,\Pi_k\big(z_0,z_0+\frac u{\sqrtk}\big)}{\Pi_k\big(z_0,z_0\big)\,\Pi_k\big(z+\frac u{\sqrtk},z_0+\frac u{\sqrtk}\big)}\\&=& e^{-|u|^2}\left\{1+ k\inv[p_2(u,0)+p_2(0,u)-p_2(0,0)-p_2(u,u)]+\cdots\right\}.,\end{eqnarray*} and therefore $a_1(u)=0$.
By \eqref{p2}, $$p_2(u,0)=p_2(0,u)= \frac12\rho(z_0)+\frac 18 R_{z_0}(u,\bar u,u,\bar u)\,,\quad p_2(0,0)=p_2(u,u)=\frac12\rho(z_0)\,.$$
Therefore, $a_2(u)=\frac14R_{z_0}(u,\bar u,u,\bar u)$.
\end{proof}

Henceforth in this paper, we assume that the \kahler potential $\phi$ satisfies \eqref{Korder4}.
\begin{lem}\label{Qapprox} For $|u|<b\sqrt{\log k}$ and $\ep>0$, we have \beq\textstyle Q_k\left(z_0+\frac u{\sqrtk},z_0\right) = G\left( e^{-|u|^2}\left[1+\frac14R_{z_0}(u,\bar u,u,\bar u)k^{-1}\right]\right) +O\left(|u|k^{-3/2+\ep}\right).\eeq \end{lem}

\begin{proof}  
By Theorem \ref{withK} (with $n=2$) and \eqref{Qk}, there is a constant $C$ (depending only on $b$ and $\ep$) such that for $|u|<b\sqrt{\log k}$, $k\gg 0$,
\beq\label{mvt-}\textstyle Q_k\big(z_0+\frac u{\sqrtk},z_0\big) = G\left(e^{-|u|^2}\left[1+\frac14R_{z_0}(u,\bar u,u,\bar u)k^{-1}\right]\right) +G'(t_k(u))\cdot O(|u|^2k^{-3/2+\ep})\,,\eeq\beq\label{mvt} \textstyle t_k(u)=e^{-|u|^2}\left[1+\frac14R_{z_0}(u,\bar u,u,\bar u)k^{-1}+C_k(u)\,|u|^2k^{-3/2+\ep}\right],\quad |C_k(u)| \le C.\eeq 
Note that we can choose $k_0$ such that for $0<|u|<b\sqrt{\log k}$  and $k\ge k_0$, we have \beq\label{C}0 <1+\frac14R_{z_0}(u,\bar u,u,\bar u)k^{-1} \pm C|u|^2k^{-3/2+\ep}\le e^{\half |u|^2}\eeq and thus $0<t_k(u)\le e^{-\half |u|^2}$ so that $G'(t_k(u))$ and $G(e^{-|u|^2}[1+\frac14R_{z_0}(u,\bar u,u,\bar u)k^{-1}])$ are well-defined for $k\ge k_0$.
To complete the proof, we must show that
\beq\label{G'} |u|\,G'(t_k(u)) = O(\sqrt{\log k}), \ \mbox{for } \ 0<|u|<b\sqrt{\log k},\ k\ge k_0.\eeq 
Let $s=e^{-\half|u|^2}$. By \eqref{mvt}--\eqref{C}, for $k\ge k_0$ we have $$|u|\,G'(t_k(u)) \le |u|\,G'(e^{-\half |u|^2})=-\frac 1{4\pi^2}\sqrt{-2\log s}\;\frac{\log(1-s)}{s}.$$ We then have $$\limsup_{u\to 0}|u|\,G'(t_k(u)) \le \limsup_{s\to 1^-}\frac {-1}{4\pi^2}\sqrt{-2\log s}\;\frac{\log(1-s)}{s} = 0, \quad\mbox{uniformly for }\ k\ge k_0.$$
Thus, $|u|\,G'(t_k(u))$ is bounded for $|u|\le 1$, $k\ge k_0$. On the other hand, $G'(e^{-\half |u|^2})$ is bounded for $|u|\ge 1$, yielding \eqref{G'}. Lemma \ref{Qapprox} then follows from \eqref{mvt-}.\end{proof}

We begin the proof of Proposition \ref{sharp2}  following the method of \cite{SZ5}: By Theorem \ref{bipot}, we have
\begin{equation}\label{int1s}\var\big(Z_{s^k},\psi\big)
= \int_{M}\ical_k\,i\ddbar\psi\;,\end{equation} where
\begin{equation}\label{int2s0}\ical_k(z_0)=\int_M
Q_k(z,z_0)\,i\ddbar \psi(z)\,, \quad \forall\ z_0\in M\;.\end{equation}

 We  write
\begin{equation}\label{psi}i\ddbar \psi =
f\,\Om_M\;,\qquad f \in\ccal^\infty_\R(M),\end{equation} so that
\begin{equation}\label{int2s}\ical_k(z_0)=\int_{ M}
Q_k(z,z_0)\,f(z)\Om_M(z)\;.\end{equation} 

We now choose local coordinates and local frame satisfying equation \eqref{Korder4} as described, for example, in \cite{Bo} or \cite[Lemma~2.7]{ LS}. By Corollary
\ref{Qdecay}, we can approximate
$\ical_k(z_0)$ by integrating \eqref{int2s} over a small ball
about $z_0$:
\begin{equation}\label{Iz}\ical_k(z_0) =   \int_{|u|\le b\sqrt{\log k}}
Q_k\left(z_0+\frac u{\sqrtk},z_0\right)\, f\left(z_0+\frac
u{\sqrtk}\right)\,
\Om_M\left(z_0+\frac
u{\sqrtk}\right)+O\left(\frac 1 {k^{m+2}}\right),
\end{equation}   where we let $b=\sqrt{m+3}$.
By \eqref{Rijkl} and \eqref{Korder4},
\beq\label{Rz0} R_{j\bar l p\bar q}(z_0) = - \frac{\d^4\phi}{\d z^j\bar\d z^l\d z^p\bar\d z^q}(z_0)\,,\eeq and thus
\begin{multline}\om\left(z_0+\frac u\sqrtk\right) =\frac i{2} \ddbar  \phi\left(z_0+\frac u\sqrtk\right)\\=\frac i{2} \sum_{j=1}^m\frac 1k du^ j\wedge d\bar u^ j -\frac i{2}\sum_{j,l,p,q} \left(\frac1{k^2}R_{j\bar lp\bar q}(z_0)u^ p\bar u^ q +O(\frac{|u|^3}{k^{5/2}})\right) du^ j\wedge d\bar u^ l\,.\label{kahlerE}\end{multline} 
Hence
\begin{equation}\Om_M\left(z_0+\frac
u{\sqrtk}\right)=\frac 1{m!}\om^m=\left[1 -\left(  \sum_{j,p,q}R_{j\bar jp\bar q}(z_0)u^ p\bar u^ q\right) k^{-1}+O(k^{-3/2+\ep})\right]\frac1{k^m}\,\nu_u\,,\label{OmegaM}\eeq
for $|u|\le b\sqrt{\log k}$, where 
\beq\label{Euc} \nu_u= (\frac i2 du^ 1\wedge d\bar u^ 1)\wedge\cdots\wedge (\frac i2 du^ m\wedge d\bar u^ m)\eeq
denotes the Euclidean volume form. We then have by \eqref{Iz}, \eqref{OmegaM} and Lemma \ref{Qapprox},
\begin{multline}\!\!\!\!\ical_k(z_0)=  \frac 1 {k^m}
 \int_{|u|\le b\sqrt{\log k}}  \left\{G\left( e^{-|u|^2}\left[1+\frac14\sum_{j,l,p,q}R_{j\bar lp\bar q}(z_0)u^ j\bar u^ lu^ p\bar u^ qk^{-1}\right]\right) +O\left(k^{-3/2+\ep}\right)\right\}\\ \times
\left\{f(z_0)+2\Re\sum_jf_j(z_0)u^ jk^{-1/2}+\sum _{j,l}\left[\Re f_{jl}(z_0)u^ ju^ l+f_{j\bar l}(z_0)u^ j\bar u^ l\right]k^{-1}+O(k^{-3/2+\ep})\right\}\\
\times \left\{1 -\left(  \sum_{j,p,q}R_{j\bar jp\bar q}(z_0)u^ p\bar u^ q\right) k^{-1}+O(k^{-3/2})\right\}\nu_u+O(k^{-m-2})\;,\label{U1}
\end{multline} where we write $$f_j=\frac{\d f}{\d z^j}\,,\quad f_{jl}=\frac{\d^2 f}{\d z^j \d z^l}\,,\quad f_{j\bar l}=\frac{\d^2 f}{\d z^j\d\bar z^l}\,.$$
Since the integral of the $O(k^{-3/2+\ep})$ terms in \eqref{U1} over the $(b\sqrt{\log k})$-ball is $O(k^{-3/2+\ep'})$, we have
\begin{multline}\ical_k(z_0)=  \frac 1 {k^m}
 \int_{|u|\le b\sqrt{\log k}}  G\left( e^{-|u|^2}\left[1+\frac14\sum_{j,l,p,q}R_{j\bar lp\bar q}(z_0)u^ j\bar u^ lu^ p\bar u^ q\;k^{-1}\right]\right) \\ \times
\left[f(z_0)+2\Re\sum_jf_j(z_0)u^ j\;k^{-1/2}+\sum _{j,l}\left[\Re f_{jl}(z_0)u^ ju^ l+f_{j\bar l}(z_0)u^ j\bar u^ l\right]k^{-1}\right]\\ 
\times \left[1 -  \sum_{j,p,q}R_{j\bar jp\bar q}(z_0)u^ p\bar u^ q\; k^{-1}\right]\nu_u+O(k^{-m-3/2+\ep'})\;.\label{U1a}\end{multline}

Since $ G(t)=O(t)$  for $0\le t\le 1$, it follows from \eqref{C}  that the integrand in \eqref{U1a} is $O\big(e^{-\half |u|^2}(1+|u|^8)\big)$ for $k\ge k_0$, $u\in\C^m$. 
Since \beq\label{decay}\int_{|u|\ge b\sqrt{\log k}}e^{-\half |u|^2}(1+|u|^r)\,\nu_u= O(k^{-b^2/2+\ep}), \qquad r\ge 0\,,\eeq  we can replace the domain of integration $\{|u|\le b\sqrt{\log k}\}$ with  $\C^m$ in \eqref{U1a} (under our assumption $b^2=m+3$).

We say that a function $F$ on $\C^m$ has {\it polynomial growth\/} if $|F(u)| = O(1+|u|^r)$ for some $r\in \R_+$. We shall use the following estimate in the proof of Proposition~\ref{sharp2}:
\begin{lem}\label{Gintlemma} Let $F\in\ccal^0(\C^m)$ such that $F$ has polynomial growth, and let $\al(u)$ be a polynomial in $(u,\bar u)$ with $\al(0)=0$, $d\al(0)=0$. Then
\begin{multline}\label{Gint}\int_{\C^m} G\left(e^{-|u|^2}[1+\al(u)k^{-1}]\right)F(u)\,\nu_u=\\ \frac 1{4\pi^2}\sum_{n=1}^\infty \frac1{n^2}\int_{\C^m} e^{-n|u|^2}[1+n\al(u)k^{-1}]F(u)\,\nu_u +O(k^{-3/2})\,.\end{multline} for $k\gg 0$.\end{lem}
\begin{proof} Since $\al(0)=0$ and $d\al(0)=0$, we can choose $k_1$ such that \beq \label{k1}1+|\al(u)|k^{-1}\le e^{|u|^2/2}\,,\quad \mbox{for }\ k\ge k_1\,,\ u\in\C^m.\eeq 
Since $|G(t)|\le G(|t|)=O(|t|)$ for $|t|\le 1$, for $k\ge k_1$  we have by \eqref{k1} \begin{eqnarray}\int_{\C^m} \left|G\left(e^{-|u|^2}[1+\al(u)k^{-1}]\right)F(u)\right|\nu_u&\le& \int_{\C^m} G\left(e^{-|u|^2}[1+|\al(u)|k^{-1}]\right)|F(u)|\,\nu_u \notag   \\&\le& \int_{\C^m} G\left(e^{-|u|^2/2}\right)|F(u)|\,\nu_u\notag\\&\le & C\int_{\C^m}e^{-|u|^2/2}|F(u)|\,\nu_u\ <\ \infty\,.\label{left}\label{Gintfinite}\end{eqnarray} Thus the left side of \eqref{Gint} is well-defined and finite, for $k\ge k_1$.
Furthermore, for $p\ge0,\ q\ge1$, 
\beq\label{sim2} \sum_{n=1}^\infty \int_{\C^m}\frac 1{n^q} e^{-n|u|^2}|u|^p\,\nu_u =\sum_{n=1}^\infty \frac 1{n^{q+m+p/2}}\int_{\C^m} e^{-|u|^2}|u|^p\,\nu_u <\infty\,\eeq
and therefore \beq\label{sim3} \left|\sum_{n=1}^\infty \frac1{n^2}\int_{\C^m} e^{-n|u|^2}[1+n\al(u)k^{-1}]F(u)\,\nu_u\right|\le \sum_{n=1}^\infty \frac1{n^2}\int_{\C^m} e^{-n|u|^2}[1+n|\al(u)|]F(u)\,\nu_u <\infty\,,\eeq for $k\ge 1$. Hence the right side of \eqref{Gint} is  also well-defined.

By \eqref{G}, \beq \int_{\C^m} G\left(e^{-|u|^2}[1+\al(u)k^{-1}]\right)F(u)\nu_u =\frac 1{4\pi^2}\sum_{n=1}^\infty \frac1{n^2}\int_{\C^m} e^{-n|u|^2}[1+\al(u)k^{-1}]^nF(u)\,\nu_u\,.\eeq
Let \begin{eqnarray*} E&:=& \int_{\C^m} G\left(e^{-|u|^2}[1+\al(u)k^{-1}]\right)F(u)\nu_u -\frac 1{4\pi^2}\sum_{n=1}^\infty \frac1{n^2}\int_{\C^m} e^{-n|u|^2}[1+n\al(u)k^{-1}]F(u)\,\nu_u\\&=&  \frac 1{4\pi^2} \sum_{n=2}^\infty  \int_{\C^m}\frac{e^{-n|u|^2}}{n^2}\sum_{j=2}^n{ n\choose j} \al(u)^jk^{-j}|F(u)|\,\nu_u\,. \end{eqnarray*}

Then
$$|E| \le \frac 1{4\pi^2} k^{-3/2}\sum_{n=2}^\infty \int_{\C^m}\frac{e^{-n|u|^2}}{n^2}\sum_{j=2}^n{ n\choose j} |\al(u)|^jk^{-j+3/2}|F(u)|\,\nu_u\,.$$
Since $-j+3/2 \le -j/4$ for $j\ge 2$, we have for $k\ge k_1^4$,
\begin{eqnarray*}|E| &\le&\frac 1{4\pi^2} k^{-3/2}\sum_{n=2}^\infty \int_{\C^m} \frac{e^{-n|u|^2}}{n^2}\sum_{j=2}^n{ n\choose j} \left(|\al(u)|k^{-1/4}\right)^j|F(u)|\,\nu_u\\&\le&  \frac 1{4\pi^2}k^{-3/2}\sum_{n=1}^\infty\int_{\C^m} \frac{e^{-n|u|^2}}{n^2}\left(1+|\al(u)|k^{-1/4}\right)^n |F(u)|\,\nu_u\\&=&\frac 1{4\pi^2}k^{-3/2}\int_{\C^m}G\left(e^{-|u|^2}(1+|\al(u)|k^{-1/4})\right)|F(u)|\,\nu_u\,.
\end{eqnarray*} By \eqref{Gintfinite},
$$\int_{\C^m}G\left(e^{-|u|^2}(1+|\al(u)|k^{-1/4})\right)|F(u)|\,\nu_u <\infty$$
for $k\ge k_1^4$, and hence $|E|=O(k^{-3/2})$.\end{proof}

We now continue the proof of Proposition \ref{sharp2}. Applying Lemma \ref{Gintlemma} with $$\al(u)=
\frac14\sum_{j,l,p,q}R_{j\bar lp\bar q}(z_0)u^ j\bar u^ lu^ p\bar u^ q\,,$$ equations \eqref{U1a}--\eqref{decay} yield
\begin{multline}\ical_k(z_0)=  \frac 1 {4\pi^2k^m}\sum_{n=1}^\infty \frac1{n^2}\int_{\C^m} e^{-n|u|^2}\left[1+\frac n4\sum_{j,l,p,q}R_{j\bar lp\bar q}(z_0)u^ j\bar u^ lu^ p\bar u^ qk^{-1}\right]\\ \times
\left[f(z_0)+2\Re\sum_jf_j(z_0)u^ jk^{-1/2}+\sum _{j,l}\left[\Re f_{jl}(z_0)u^ ju^ l+f_{j\bar l}(z_0)u^ j\bar u^ l\right]k^{-1}\right]\\
\times \left[1 -  \sum_{j,p,q}R_{j\bar jp\bar q}(z_0)u^ p\bar u^ q\; k^{-1}\right]\nu_u+O\left(\frac 1{k^{m+3/2-\ep}}\right).\label{U2}\end{multline}

Then after gathering terms, we have \begin{multline}
\ical_k(z_0)=
\frac 1 {4\pi^2k^m}\sum_{n=1}^\infty \frac1{n^2}\int_{\C^m} e^{-n|u|^2}\left\{f(z_0)+ 2\Re\sum_jf_j(z_0)u^ jk^{-1/2}\right.\\+\left(\frac n4\sum_{j,l,p,q}R_{j\bar lp\bar q}(z_0)u^ j\bar u^ lu^ p\bar u^ q-  \sum_{j,p}R_{j\bar jp\bar p}(z_0)u^ p\bar u^ q\right)     f(z_0)k^{-1}\\+\left.\sum _{j,l}\left[\Re f_{jl}(z_0)u^ ju^ l+f_{j\bar l}(z_0)u^ j\bar u^ l\right]k^{-1}\right\}\nu_u+O\left(\frac 1{k^{m+3/2-\ep}}\right).\label{U3}\end{multline}
By making a change of variable $u^ j\mapsto e^{i\theta}u^ j$ in \eqref{U3}  for a fixed index $j$ and noting that the volume form is invariant under this transformation, one sees that  terms where $u^ j$ is not paired with  $\bar u^ j$ have vanishing integrals. So we obtain
\begin{multline}
\ical_k(z_0)=
\frac 1 {4\pi^2k^m}\sum_{n=1}^\infty \frac1{n^2}\int_{\C^m} e^{-n|u|^2}\Bigg\{f(z_0)+\\\left(\left[\frac n4\sum_{j,l,p,q}R_{j\bar lp\bar q}(z_0)u^ j\bar u^ lu^ p\bar u^ q-  \sum_{j,p}R_{j\bar jp\bar p}(z_0)|u^ p|^2\right]     f(z_0)+\sum _{p}f_{p\bar p}(z_0)|u^ p|^2\right)k^{-1}\Bigg\}\nu_u\\+O\left(\frac 1{k^{m+3/2-\ep}}\right).\label{U4}\end{multline}
With the change of variables $u=\frac 1{\sqrt n}v$, we have
\begin{multline}\ical_k(z_0)=
\frac {\zeta(m+2)} {4\pi^2}\left(\int_{\C^m}e^{-|v|^2}\nu_v\right) f(z_0)k^{-m}\\
+\frac {\zeta(m+3)} {4\pi^2}\left(\int_{\C^m}e^{-|v|^2}\left[ \frac 14\sum_{j,l,p,q}R_{j\bar lp\bar q}(z_0)v^ j\bar v^ lv^ p\bar v^ q-  \sum_{j,p}R_{j\bar jp\bar p}(z_0)|v^ p|^2\right]f(z_0)\right.\\\left. + \sum _{p}f_{p\bar p}(z_0)|v^ p|^2\right)  \nu_v\, k^{-m-1}+O\left(k^{-m-3/2+\ep}\right).
\label{U5}\end{multline}

By the Wick formula and \eqref{Rz0},
\begin{eqnarray*}\frac 1{\pi^m}\int_{\C^m}\left(\sum_{j,l,p,q}R_{j\bar lp\bar q}(z_0)v^ j\bar v^ lv^ p\bar v^ q \right)e^{-|v|^2}\nu_v&=&\sum_{j,l,p,q}R_{j\bar lp\bar q}(z_0)(\delta^j_{l}\delta^p_{q}+\delta^j_{q}\delta^p_{l})\\&=&\sum_{j,p}\left[R_{j\bar jp\bar p}(z_0)+R_{j\bar pp\bar j}(z_0)\right]\\&=&2\sum_{j,p}R_{j\bar jp\bar p}(z_0)\ =\ 2\rho(z_0)\,,\end{eqnarray*} where we recall that $\rho(z_0)$ denotes the scalar curvature at $z_0$.
Thus
\begin{multline}\ical_k(z_0)=\frac{\pi^{m-2}}{4}\zeta(m+2)f(z_0)k^{-m}\\+\frac{\pi^{m-2}}{4}\zeta(m+3)\left[-\frac 12\rho(z_0)f(z_0) + \sum _{p}f_{p\bar p}(z_0)\right]k^{-m-1}+O\left(k^{-m-3/2+\ep}\right).
\label{U6}\end{multline}
Corollary \ref{Qdecay} and Theorem \ref{near} guarantee that the remainder estimate $O(k^{-m-3/2+\ep})$ in \eqref{U6} is uniform over  $z_0\in M$.

Note that $\sum _{p}f_{p\bar p}(z_0)=-\bar\d^*\bar\d f(z_0)$ since the $z^p$ are normal coordinates at $z_0$. Hence
by \eqref{int1s} and \eqref{U6}, we have
\begin{multline}\var\big(Z_{s^k},\psi\big) = \frac{\pi^{m-2}}{4}\zeta(m+2)\left(\int_Mf^2\Om_M\right)k^{-m}\\-\frac{\pi^{m-2}}{4}\zeta(m+3)\int_M\left[\frac 12\rho f^2 + f\bar\d^*\bar\d f\right]\Om_Mk^{-m-1}+O(k^{-m-3/2+\ep})\,.
\label{Var}\end{multline}
Since $f=*i\ddbar\psi$, we have 
\beq\label{f^2}f^2=|\ddbar\psi|^2,\qquad \int_Mf^2\Om_M = \int_M|\ddbar\psi|^2\Om_M= \|\ddbar\psi\|_2^2\,.\eeq
Furthermore, \beq\label{3rd} \int_Mf\bar\d^*\bar\d f\,\Om_M= \langle \bar\d^*\bar\d f, f\rangle =  \|\bar\d f\|^2_2= \|\bar\d (*\ddbar\psi)\|_2^2= \|\d^*\ddbar\psi\|^2_2.\eeq
The formula of Proposition~\ref{sharp2} follows from
\eqref{Var}--\eqref{3rd}.\qed
 
\section{Proof of Theorem \ref{sharp}.}\label{theproof}

To complete the proof of Theorem \ref{sharp}, it suffices to show that
\beq\label{asympt} \var(Z_{s^k},\psi) = A_0 k^{-m}+A_1k^{-m-1}+\cdots + A_pk^{-m-p}+O(k^{-m-p-1})\,,\eeq for $p\ge 2$.
We shall verify \eqref{asympt} by generalizing Lemmas \ref{Qapprox}--\ref{Gintlemma}. First we have
\begin{lem}\label{Qapproxp} Let $p\ge2$, $\ep>0$. For $|u|<b\sqrt{\log k}$, we have \beq\textstyle Q_k\left(z_0+\frac u{\sqrtk},z_0\right) = G\left( e^{-|u|^2}\left[1+\sum_{j=2}^pa_j(u)k^{-j/2}\right]\right) +O\left(|u|k^{-(p+1)/2+\ep}\right),\eeq where  the $a_j(u)$ are as in Theorem~\ref{withK} and $a_2(u)=\frac14R_{z_0}(u,\bar u,u,\bar u)$. \end{lem}
\begin{proof} Repeat the proof of Lemma~\ref{Qapprox} with $\frac14R_{z_0}(u,\bar u,u,\bar u)k\inv$ replaced by  $\sum_{j=2}^pa_j(u)k^{-j/2}$ and with $k^{-3/2+\ep}$ replaced by $k^{-(p+1)/2+\ep}$.\end{proof}

To state the generalization of Lemma \ref{Gintlemma}, we use the following notation: For integers $j\ge 2$, we let $$\textstyle \pcal(j)= \left\{(r_2,\dots,r_p)\in\N^{p-1}:\sum_{\la=2}^p \la\,r_\la=j\right\},$$ where $\N$ denotes the non-negative integers.
For ${\bf r}= (r_2,\dots,r_p)\in\N^{p-1}$ and $n\in\Z_+$, we let
\beq\label{PI} {n \choose {\bf r}}=\frac{n(n-1)\cdots (n-|{\bf r}|+1)}{r_2!\cdots r_p!} =\left\{\begin{array}{cl}\frac{n!}{r_2!\cdots r_p!(n-|{\bf r}|)!} & \quad \mbox{for }\ n\ge |{\bf r}|\\ 0 & \quad \mbox{for }\ n\in\{0,1,\dots,|{\bf r}|-1\}\end{array}\right.\eeq denote the multinomial coefficients, where $|{\bf r}|=r_2+\cdots +r_p$. We note the polynomial identity
\beq\label{poly}\Big(1+\sum_{j=2}^pC_jx^j\Big)^n=1+\sum_{j=2}^{np}B_{pj}(C_2,\dots,C_p;n)x^j\,, 
\quad B_{pj}(C_2,\dots,C_p;n)=\sum_{{\bf r}\in\pcal(j)}{n \choose {\bf r}}\prod_{\la=2}^p C_{\la}^{r_\la}\,.\eeq

The following generalizes  Lemma \ref{Gintlemma}:
\begin{lem}\label{plemma} Let $F\in\ccal^0(\C^m)$ such that $F$ has polynomial growth. Then for $p\ge 2$,
\begin{multline}\label{Gintp}\int_{\C^m} G\left(e^{-|u|^2}\left[1+\sum_{j=2}^pa_j(u)k^{-j/2}\right]\right)F(u)\,\nu_u=\\ \frac 1{4\pi^2}\sum_{n=1}^\infty \frac1{n^2}\int_{\C^m} e^{-n|u|^2}\left[1+\sum_{j=2}^pb_{pj}(u,n)k^{-j/2}\right]F(u)\,\nu_u +O(k^{-p/2-1/4})\,,\end{multline}  
where \beq\label{bj} b_{pj}(u,n)=B_{pj}\left(a_2(u),\dots,a_p(u);n\right)\,.\eeq
\end{lem}

\begin{proof} Fix $p\ge 2$. Since the polynomials $a_j$ have no constant or linear terms, we can choose $k_0$ such that $$1+\sum_{j=2}^p|a_j(u)|k^{-j/2}\le e^{|u|^2/2}\,,\  \mbox{for } k\ge k_0,\ u\in\C^m.$$
Since $|G(t)|\le G(|t|)=O(|t|)$ for $|t|\le 1$, we then have \begin{multline}\int_{\C^m}\left| G\left(e^{-|u|^2}\left[1+\sum_{j=2}^pa_j(u)k^{-j/2}\right]\right)F(u)\right|\nu_u\\ \le \int_{\C^m} G\left(e^{-|u|^2}\left[1+\sum_{j=2}^p|a_j(u)|k^{-j/2}\right]\right)|F(u)|\,\nu_u \le \int_{\C^m} G\left(e^{-|u|^2/2}\right)|F(u)|\,\nu_u <\infty\label{GGG}\end{multline}
and thus the left side of \eqref{Gintp} is well-defined and finite, for $k\ge k_0$.

By \eqref{G} and \eqref{poly}, \begin{eqnarray} G\left(e^{-|u|^2}\left[1+\sum_{j=2}^pa_j(u)k^{-j/2}\right]\right)&=&\frac 1{4\pi^2}\sum_{n=1}^\infty \frac1{n^2} e^{-n|u|^2}\left[1+\sum_{j=2}^pa_j(u)k^{-j/2}\right]^n\notag\\&=&\frac 1{4\pi^2}\sum_{n=1}^\infty \frac1{n^2} e^{-n|u|^2}\left[1+\sum_{j=2}^{np}b_{pj}(u,n)k^{-j/2}\right],\label{bjsum}\end{eqnarray}for $k\ge k_0$, where $b_{pj}(u,n)$ is given by \eqref{bj}.

Furthermore, since ${n \choose {\bf r}}\ge 0$ for $n\in\Z_+$, we have  $$|b_{pj}(u,n)|\le  \sum_{{\bf r}\in\pcal(j)}{n \choose {\bf r}}\prod_{\la=2}^p |a_{\la}(u)|^{r_\la}= B_{pj}(|a_2(u)|,\dots,|a_p(u)|;n)\,, \quad \mbox{for }\ n\in\Z_+\,,$$
and thus $$1+\sum_{j=2}^{np}|b_{pj}(u,n)|k^{-j/2} \le 1+\sum_{j=2}^{np}B_{pj}(|a_2(u)|,\dots,|a_p(u)|;n)k^{-j/2}=\left[1+\sum_{j=2}^p|a_j(u)|k^{-j/2}\right]^n.$$ 
Hence by \eqref{GGG}, \begin{multline}\sum_{n=1}^\infty \int_{\C^m} \frac{e^{-n|u|^2}}{n^2}\left[1+\sum_{j=2}^{np}|b_{pj}(u,n)|k^{-j/2}\right]|F(u)|\nu_u \\ \le \sum_{n=1}^\infty \int_{\C^m} \frac{e^{-n|u|^2}}{n^2}\left[1+\sum_{j=2}^p|a_j(u)|k^{-j/2}\right]^n|F(u)|\, \nu_u\\ = 4\pi^2\int_{\C^m} G\left(e^{-|u|^2}\left[1+\sum_{j=2}^p|a_j(u)|k^{-j/2}\right]\right)|F(u)|\,\nu_u <\infty\,,\label{Bsum}
\end{multline} for $k\ge k_0$.  In particular, the right side of \eqref{Gintp} is also well-defined and finite.

As before, let \begin{eqnarray*}E&:=&\int_{\C^m} G\left(e^{-|u|^2}\left[1+\sum_{j=2}^pa_j(u)k^{-j/2}\right]\right)F(u)\,\nu_u\\ &&\qquad-\frac 1{4\pi^2}\sum_{n=1}^\infty \frac1{n^2}\int_{\C^m} e^{-n|u|^2}\left[1+\sum_{j=2}^pb_{pj}(u,n)k^{-j/2}\right]F(u)\,\nu_u\\&=&\frac 1{4\pi^2}\sum_{n=2}^\infty \frac1{n^2}\int_{\C^m} e^{-n|u|^2}\sum_{j=p+1}^{np}b_{pj}(u,n)k^{-j/2}F(u)\,\nu_u.\end{eqnarray*}

Therefore  $$|E| \le\frac 1{4\pi^2}k^{-p/2-1/4}\sum_{n=2}^\infty \frac1{n^2}\int_{\C^m} e^{-n|u|^2}\sum_{j=p+1}^{np}|b_{pj}(u,n)|\,k^{(p-j)/2+1/4}\,|F(u)|\,\nu_u\,.$$ Since $$p-j+\frac12\, \le\, \frac{-j}{2p+2}\quad\mbox{for }\ j\ge p+1\,,$$
$$  |E| \le 
\frac 1{4\pi^2}k^{-p/2-1/4}\sum_{n=2}^\infty \frac1{n^2}\int_{\C^m} e^{-n|u|^2}\sum_{j=p+1}^{np}|b_{pj}(u,n)|\,k^{-j/(4p+4)}\,|F(u)|\,\nu_u\,.$$ 
By \eqref{Bsum}, $$\sum_{n=2}^\infty \frac1{n^2}\int_{\C^m} e^{-n|u|^2}\sum_{j=p+1}^{np}|b_{pj}(u,n)|\,k^{-j/(4p+4)}\,|F(u)|\,\nu_u <\infty$$ for $k\ge k_0^{2p+2}$, and thus $|E|=O(k^{-p/2-1/4})$.\end{proof}

\begin{proof}[Continuation of the proof of Theorem~\ref{sharp}]
Fix $p\ge 2$ and let $b=\sqrt{m+p+1}$. Recalling \eqref{psi}, we have an expansion of the form
\begin{multline}\label{series} i\ddbar\psi\left(z_0+\frac u\sqrtk\right)=f\left(z_0+\frac u\sqrtk\right)\Om_M\left(z_0+\frac u\sqrtk\right)\\=\left[f(z_0)+2\Re\sum_{q=1}^mf_q(z_0)u^qk^{-1/2}+\sum_{j=2}^p\be_j(u)k^{-j/2}+O(k^{-p/2-1/2+\ep})\right]\frac 1{k^m}\nu_u\,,\end{multline} for $|u|\le b\sqrt{\log k}$, where $\be_j(u)$ is a homogeneous polynomial in $(u,\bar u)$ of degree $j$.

Repeating the derivation of \eqref{U1a} using \eqref{Iz}, \eqref{series} and Lemma~\ref{Qapproxp}, we obtain
\begin{multline} \ical_k(z_0) =\frac 1{k^m}\int_{\C^m}G\left(e^{-|u|^2}\left[1+\sum_{j=2}^pa_j(u)k^{-j/2}\right]\right)\\ \times \left[f(z_0)+2\Re\sum_{q=1}^mf_q(z_0)u^qk^{-1/2}+\sum_{j=2}^p\be_j(u)k^{-j/2}\right]\nu_u +O(k^{-m-p/2-1/4})\,.\label{U1p}\end{multline}
Here we set $\ep=1/4$ and again  used \eqref{decay}. By \eqref{U1p} and Lemma~\ref{plemma}, we then have 
\begin{multline}\ical_k(z_0)=  \frac 1 {4\pi^2k^m}\sum_{n=1}^\infty \frac1{n^2}\int_{\C^m} e^{-n|u|^2}\left[1+\sum_{j=2}^pb_{pj}(u,n)k^{-j/2}\right]\\ \times
\left[f(z_0)+2\Re\sum_{q=1}^mf_q(z_0)u^qk^{-1/2}+\sum_{j=2}^p\be_j(u)k^{-j/2}\right]\nu_u+O\left(\frac 1{k^{m+p/2+1/4}}\right).\label{U2p}\end{multline}
Furthermore by Theorem~\ref{withK} and \eqref{bj}, \beq\label{bjl} b_{pj}(u,n)=  \sum_{{\bf r}\in\pcal(j)}{n \choose {\bf r}}\prod_{\la=2}^p a_{\la}(u)^{r_\la}=\sum_{l=0}^{\lfloor j/2\rfloor}b_{pjl}(u)n^l\qquad \big({\bf r}=({r_2},\dots ,{r_p})\big)\,,\eeq where  $b_{pjl}(u)$ is a polynomial in $(u,\bar u)$ of degree $\le 5j$ and of the same parity as $j$. 
(The highest power of $n$ in $b_{pj}$ is  $\lfloor j/2\rfloor$ since  $j=2r_2+\cdots+pr_p\ge 2|{\bf r}|= 2\deg_{\Q[n]} {n \choose {\bf r}}$ for  ${\bf r}\in\pcal(j)$. However, this bound is not needed in the proof.)

Since the polynomials $a_\la(u)$ contain only terms of degree 2 or higher in $(u,\bar u)$, it follows that $\prod_{\la=2}^p a_{\la}(u)^{r_\la}$ contains only terms of degree $\ge2|{\bf r}|=2\deg_{\Q[n]}{n \choose {\bf r}}$. It then follows from \eqref{bjl} that $b_{pjl}(u)$ is of the form $$b_{pjl}(u)=\sum_{q=2l}^{5j}\!\raisebox{6pt}{$\prime$}b_{pjlq}(u)\,,$$ where $b_{pjlq}(u)$ is homogeneous of degree $q$ in $(u,\bar u)$, and $\sum\!\raisebox{4pt}{$\prime$}$ denotes the sum with $q\equiv j (2)$.
Thus \eqref{U2p} can be written in the form
\beq\ical_k(z_0)=  \frac 1 {4\pi^2k^m}\sum_{n=1}^\infty \int_{\C^m} \frac {e^{-n|u|^2}}{n^2}
\sum_{j=0}^{2p}\sum_{l=0}^{\lfloor j/2\rfloor}\sum_{q=2l}^{5j}\!\raisebox{6pt}{$\prime$}S_{pjlq}(u)n^lk^{-j/2}\;\nu_u+O\left(\frac 1{k^{m+p/2+1/4}}\right),\label{U3p}\eeq
where    $S_{pjlq}(u)$ is a homogeneous polynomial in $(u,\bar u)$ of degree $q$. (In particular, $S_{p000}(u)=f(z_0)$.)

As before, we make the change of variables $u=\frac 1{\sqrt n} v$ so that \eqref{U3p} becomes
\begin{eqnarray}\ical_k(z_0) &=&  \frac 1 {4\pi^2k^m}\sum_{n=1}^\infty \int_{\C^m} \sum_{j=0}^{2p}\sum_{l=0}^{\lfloor j/2\rfloor}\sum_{q=2l}^{5j}\!\raisebox{6pt}{$\prime$}\,\frac {e^{-|v|^2}S_{pjlq}(v)}{n^{2+m+q/2-l}}\,
k^{-j/2}\;\nu_v+O\left(\frac 1{k^{m+p/2+1/4}}\right)\notag\\ &=&  \frac 1 {4\pi^2k^m} \sum_{j=0}^{2p}
k^{-j/2}\sum_{l=0}^{\lfloor j/2\rfloor}\sum_{q=2l}^{5j}\!\raisebox{6pt}{$\prime$}\,\zeta(2+m+q/2-l)\int_{\C^m} {e^{-|v|^2}S_{pjlq}(v)}\;\nu_v\notag\\ &&\hspace{3.5in}+O\left(\frac 1{k^{m+p/2+1/4}}\right).\label{U4p}\end{eqnarray} 
Since the terms  of the sum in \eqref{U4p} with $q$ odd (in which $S_{pjlq}(v)$ has odd degree) vanish, and since $q$ and $j$ have the same parity in the sum,  \eqref{U4p} reduces to a sum over $q$ and $j$ even.  I.e., substituting $j=2\tilde j,\,q=2\tilde q$, we have
\beq\ical_k(z_0) = \frac 1 {4\pi^2k^m} \sum_{\tilde j=0}^{p}
k^{-\tilde j}\sum_{l=0}^{\tilde j}\sum_{\tilde q=l}^{5\tilde j}\zeta(2+m+\tilde q-l)\int_{\C^m} {e^{-|v|^2}S_{p(2\tilde j)l(2\tilde q)}(v)}\;\nu_v+O\left(\frac 1{k^{m+p/2+1/4}}\right).\label{U5p}\eeq
Finally, by \eqref{int1s} and \eqref{U5p} with $p$ replaced by $2p+2$, we obtain the asymptotic expansion \eqref{asympt} of Theorem \ref{sharp}. Proposition~\ref{sharp2} provides the values of the leading and sub-leading terms of the expansion.\end{proof}

\section{Remarks}

By further refining the local coordinate condition \eqref{Korder4}, one can obtain information about additional terms of the asymptotic expansion \eqref{asympt}.  I.e., for $n\ge 4$ one can choose holomorphic {\it Bochner coordinates\/} (also called {\it K-coordinates}) of order $n$ at $z_0$, which  satisfy: 
 
\begin{equation} \label{bochnerk}  \phi(z_0+z) = |z|^2 +\!\!\! \sum_{\begin{smallmatrix}|J|\ge 2, \; |K| \ge 2\\|J|+|K|\le n\end{smallmatrix}}\!\!\! a_{JK} z^J \bar z^K \;+\;O(|z|^{n+1})\,.\end{equation} 
See \cite{Bo,LS}. Equation \eqref{Korder4} describes Bochner coordinates of order 4. It was shown in \cite[Th.~2.8]{LS} that with Bochner coordinates of order $r+2$, the polynomials $p_r(u,v)$ of Theorem~\ref{near} are curvature invariants (and are of degree  $\le 2r$ instead of $5r$).   
It then follows by tracing through the proof of Theorem~\ref{sharp} in Section~\ref{theproof} that the coefficients $A_j$ in Theorem~\ref{sharp} are integrals involving   curvature invariants and $\ddbar\psi$.  To state this precisely, we introduce the following definition:
\begin{defin}  Let $M$ be  a Riemannian manifold and suppose that $f\in\ccal^\infty(M)$.  An\break  {\em $f$-curvature invariant of $M$} is a scalar field (smooth function) on $M$ that  is a  contraction of tensor products of $f$ and its derivatives and the curvature tensor of $M$ and its  derivatives.\end{defin}
The proof of Theorem~\ref{sharp}  yields the following result:

\begin{theo}\label{general}  Let $(L,h)\to (M,\om)$ be as in Theorem
\ref{sharp}. Suppose that $\psi\in \dcal^{m-1,m-1}_\R(M)$ and let $i\ddbar\psi=f\Om_M$, where $\Omega_M$ is the volume form on $M$. Then the coefficients $A_j$ of the asymptotic expansion \eqref{expansion} of $\var\big(Z_{s^k},\psi\big)$
are of the form $A_j=\int_M\acal_j\,\Omega_M$, where $\acal_j$ is a linear combination of $f$-curvature invariants of $M$.\end{theo}
Formulas for the polynomials  $p_3$ and $p_4$ of Theorem~\ref{near} were also given in \cite{LS} using Bochner coordinates, and these can be used together with \eqref{U2p} and \eqref{U5p} to obtain a formula for $A_2$. (The integral in \eqref{U5p} can be evaluated using the Wick formula, and $5\tilde j$ can be replaced by $2\tilde j$.)

In addition to the linear statistics studied here,  the following ``number statistics" are also of interest: For a domain $U\subset M$ with smooth boundary, we let $\ncal^U_k$ denote the number of simultaneous zeros in $U$ of  $m$ independent Gaussian holomorphic sections of $L^k\to M$. It was shown in \cite{SZa} that the variance of $\ncal^U_k$ has the asymptotics
\beq\label{number} \var\left(\ncal^U_k\right)=\nu_{mm}\mbox{Vol}(\d U)\,k^{m-\frac12}+O(k^{m-1+\ep})\,,\eeq
where $\nu_{mm}$ is a universal constant (given explicitly in \cite{SZa}). In particular, for dimension $m=1$, we have
\beq\label{number1} \var\left(\ncal^U_k\right)= \frac{\zeta(3/2)}{8\pi^{3/2}}\,\mbox{Length}(\d U)\,k^{1/2}+O(k^\ep)\,.\eeq
The analogy with linear statistics leads to the conjecture that $\var(\ncal^U_k)$ has an asymptotic expansion.  In dimension 1, an asymptotic expansion should follow by the methods of this paper.  The higher dimensional case requires a more complicated analysis.

\end{document}